\documentclass[12pt]{article}
\usepackage{a4wide}
\usepackage{amsmath, amsthm}
\usepackage{amssymb}
\usepackage{graphicx}
\usepackage{xfrac}
\usepackage{relsize}
\usepackage{xcolor}

\newcommand{\hidden}[1]{}

\theoremstyle{plain}
\newtheorem{thm}{Theorem}
\newtheorem{cor}[thm]{Corollary}
\newtheorem{lem}[thm]{Lemma}
\newtheorem{prop}[thm]{Proposition}

\theoremstyle{definition}

\newtheorem{rem}{Remark}

\newcommand{\ca}[1]{\mathcal{#1}}

\newcommand{\beq}{\begin{equation}}
\newcommand{\eq}{\end{equation}}

\newcommand{\N}{\mathbb{N}}

\newtheorem{phil}{Theorem (Philipp)\!\!\!}
\newtheorem{hjk}{Theorem (Haynes, Jensen \& Kristensen)\!\!}
\newtheorem{thdel}{Theorem (Davenport, Erd\H{o}s \& LeVeque)\!\!}
\newtheorem{kqr}{Theorem (Kaufman \& Queff\'elec-Ramar\'e)\!\!}

\def\Bad{{\rm {\bf Bad}}}

\begin{document}

\title{ The discrepancy of $(n_kx)_{k=1}^{\infty}$ with respect to certain probability measures }

\author{ Niclas Technau \footnote{Research Supported by EPSRC Programme Grant EP/J018260/1 }  \\ {\small\sc (York) }  \and
Agamemnon Zafeiropoulos \footnote{Research supported by FWF project Y-901} \\ {\small\sc (TU Graz) }  
}

\date{}

\maketitle

\begin{abstract} \noindent Let $(n_k)_{k=1}^{\infty}$ be a lacunary sequence of integers. We show that if $\mu$ is a probability measure on $[0,1)$ such that $|\widehat{\mu}(t)|\leq c|t|^{-\eta}$, then for $\mu$-almost all $x$, the discrepancy $D_N(n_kx)$ satisfies 
\begin{equation*}
\frac{1}{4} \leq \limsup_{N\to\infty}\frac{N D_N(n_kx)}{\sqrt{N\log\log N}} \leq C  
\end{equation*}   
for some constant $C>0$. This proves a conjecture of Haynes, Jensen and Kristensen and allows an  improvement on their previous result relevant to an inhomogeneous version of the Littlewood Conjecture. 
\end{abstract}

\renewcommand{\baselinestretch}{1}
\parskip=1ex

\section{Introduction }
Let $(x_n)_{n=1}^{\infty}$ be a sequence of numbers in the unit interval $[0,1)$. We define the {\em $N$-discrepancy} of the sequence $(x_n)_{n=1}^{\infty}$ to be 
\begin{equation*}
D_N(x_n) = \sup_{0\leq \alpha<\beta<1}\left| \frac{1}{N}Z(N;\alpha,\beta) - (\beta - \alpha)  \right|
\end{equation*}
where $Z(N;\alpha,\beta):= \# \{1\!\leq k\!\leq\!N \, : \, \alpha\leq x_k \leq \beta \} $. A sequence $(x_n)_{n=1}^{\infty}$ is by definition uniformly distributed $\hspace{-2mm}\mod 1$ if and only if $D_N(x_n)\to 0$ as $N\to\infty$. Regarding the order of magnitude of the discrepancy of arbitrary sequences, Schmidt \cite{schmidt} has shown that the discrepancy of any sequence $(x_n)_{n=1}^{\infty}\subseteq [0,1)$ satisfies 
\begin{equation*}
D_N(x_n) \geq c\, \frac{\log N}{N} \hspace{4mm} \text{for inf. many } N=1,2,\ldots
\end{equation*}
where $c>0$ is an absolute constant; thus the discrepancy of an arbitrary sequence cannot tend to $0$ arbitrarily fast. \vspace{2mm} 

\noindent A case of particular interest is the discrepancy of $(n_kx)_{k=1}^{\infty}$, where $(n_k)_{k=1}^{\infty}$ is lacunary and $x\in [0,1)$. Recall that a sequence $(n_k)_{k=1}^{\infty}$ of positive integers is called {\em lacunary} if there exists some constant $q > 1$ such that 
\begin{equation} \label{lacunary}
\frac{n_{k+1}}{n_k} \, \geq \, q \, , \hspace{5mm} k=1,2,\ldots 
\end{equation}
It is well known that whenever \eqref{lacunary} holds, the sequence of functions $(e(n_kx))_{k=1}^{\infty}$ behaves like a sequence of independent random variables (here and in what follows we use the notation $e(x)=e^{2\pi ix}$); for more details we refer to the survey papers \cite{aistberkes2,kac}. One instance of this phenomenon is a result of Erd\H{o}s and Gal \cite{erdosgal} stating that 
\begin{equation}
\limsup_{N\to\infty}\frac{\left|\sum\limits_{k=1}^{N}e(n_kx) \right|}{\sqrt{N\log\log N}} = 1 \hspace{5mm} \text{ for Lebesgue-almost all } x\in[0,1),
\end{equation}
which is an analogue of the Law of the Iterated Logarithm for sequences of independent random variables. Regarding the precise order of magnitude of $D_N(n_kx)$ in that case, it had been conjectured that $D_N(n_kx)$ also satisfies a Law of the Iterated Logarithm, and this was shown to be true by Philipp \cite{philipp}.

\begin{phil}{\em Let $(n_k)_{k=1}^{\infty}$ be a lacunary sequence of integers such that \eqref{lacunary} is satisfied. Then for Lebesgue-almost all $x\in [0,1)$ we have 
\begin{equation} \label{philequation}
\frac{1}{4} \, \leq \, \limsup_{N\to\infty} \frac{ND_N(n_kx)}{\sqrt{N\log\log N}} \, \leq \, C_q \, ,
\end{equation} }
where $C_q\leq 166+664(q^{1/2}-1)^{-1} $ is a constant which depends on $q>1$. 
\end{phil}
\noindent  This is in accordance with the Chung-Smirnov Law of the Iterated Logarithm, which states that for any sequence $(X_n)_{n=1}^{\infty}$ of independent random variables, uniformly distributed on $[0,1)$, we have
$$ \limsup_{N\to\infty}\frac{ND_N(X_1,X_2,\ldots,X_N)}{\sqrt{2N\log\log N }} = \frac{1}{2}  $$
with probability $1$ (see \cite[p.504]{sw}), thus further  indicating the resemblance of $(n_kx)_{k=1}^{\infty}$ with a sequence of independent random variables. The exact value of the limsup in \eqref{philequation} for specific choices of the sequence $(n_k)_{k=1}^{\infty}$ has been calculated by Fukuyama et.al. in a series of papers \cite{fukuyama, fukuyama2, fuku2, fuku3, fuku4}.

\noindent In the present article we examine whether Philipp's metrical result can be generalised for measures which are supported on several fractal subsets of the unit interval. We focus our attention on probability measures $\mu$ such that their Fourier transform defined by
$$ \widehat{\mu}(t) = \int e^{2\pi i xt}\mathrm{d}\mu(x), \hspace{5mm} t\in\mathbb{R} $$
has a prescribed decay rate. In the results to follow, we assume that the Fourier transform of $\mu$ has a polynomial decay rate, that is, an asymptotic relation of the form 
\begin{equation} \label{mudecay}
|\widehat{\mu}(t)| \, \ll \, |t|^{-\eta}, \hspace{5mm} |t|\to \infty
\end{equation} 
holds for some constant $\eta>0$. The connection of the decay rate of the Fourier transform of $\mu$ with distribution properties is not unexpected, in view of the following theorem of Davenport, Erd\H{o}s and LeVeque.
 
\begin{thdel} {\em Let $\mu$ be a probability  measure supported on $[0,1]$ and $(q_n)_{n=1}^{\infty} $ be a sequence of natural numbers. If \vspace{-1mm}
\begin{equation} \label{del}
\sum_{N=1}^\infty \frac{1}{N^3}  \; \sum_{m,n=1}^N  \widehat{\mu}
(h(q_m-q_n)) \; < \; \infty  \; \vspace{-1mm}
\end{equation}
 for all integers $h\ne 0$, then the sequence $
(q_nx)_{n\in \N} $ is uniformly distributed modulo one  for $\mu$--almost all $x \in [0,1)$.}
\end{thdel}


\noindent The main result of this paper is the following.

\begin{thm} \label{thm1}
Let $(n_k)_{k=1}^{\infty}$ be a lacunary sequence of integers satisfying \eqref{lacunary}. Assume $\mu$ is a probability measure on $[0,1)$ such that \eqref{mudecay} holds for some $\eta>0.$ Then the discrepancy $D_N(n_kx)$ satisfies 
\begin{equation} \label{thm1eq}
\frac{1}{4} \leq \limsup_{N\to\infty} \frac{ND_N(n_kx)}{\sqrt{N\log\log N}} 
\leq C  \hspace{5mm} \text{ for $\mu$-almost all }x\in[0,1), 
\end{equation}
where the constant $C>0$ only depends on the value of $q>1$ as in \eqref{lacunary}. Additionally $C\leq 166+664(q^{1/2}-1)^{-1}.$ 
\end{thm}

An application of Theorem \ref{thm1} is an improvement of a result of Haynes, Jensen and Kristensen in \cite{hjk} relevant to an inhomogeneous version of Littlewood's conjecture, which is the statement that for all $\alpha,\beta\in\mathbb{R}$, we have $\liminf q\|q\alpha\|\|q\beta\|=0.$ This is clearly the case when $\alpha$ or $\beta$ is an element of the set $\Bad:=\{x\in [0,1): \liminf q\|qx\|>0 \}$ of {\em badly approximable} numbers. The result proved in \cite{hjk} is the following: 
\begin{hjk}{\em Fix $\varepsilon>0$ and a sequence $(\alpha_i)_{i=1}^{\infty}\subseteq \Bad$. Then there exists a set $G\subseteq \Bad$ of Hausdorff dimension $\dim G=1$, such that for all $\beta\in G$ the following holds: 
$$ q\|q\alpha_i\|\|q\beta-\gamma\|< 1/(\log q)^{\frac{1}{2}-\varepsilon} \hspace{4mm} \text{ for inf. many } q=1,2,\ldots $$
for all $i\geq 1$ and all $\gamma\in\mathbb{R}$.}
\end{hjk}
The proof of the result in \cite{hjk} relies on a metric discrepancy estimate with respect to certain probability measures supported on subsets of the set $\Bad$. More precisely, if $F_N:=\{x=[a_1,a_2,\ldots] \in[0,1) : a_n\leq N \text{ for all } n\geq 1\}$ is the set of $x\in [0,1)$ such that the partial quotients in the continued fraction expansion of $x$ are at most equal to $N$, a theorem of Kaufman \cite{Kaufman}, later improved by Queff\'elec and Ramar\'e \cite{QR}, states that the sets $F_N, \, N\geq 2$ support probability measures with two key properties:

\begin{kqr}{\em Let $N\geq 2$. If $\varepsilon>0$ and $\frac{1}{2}<\delta<\dim F_N$, then the set $F_N$ supports a probability measure $\mu=\mu(N,\delta,\varepsilon)$ with the following properties:\vspace{-2.5mm}
\begin{description}
\item{(i)} $\mu (I) \leq c_1|I|^{\delta}$ for any interval $I\subseteq [0,1)$, and  \vspace{-1.1mm}       
\item{(ii)} $|\widehat{\mu}(t)| \leq c_2(1+|t|)^{-\eta+8\varepsilon}$\, for all $t\in\mathbb{R},$ where $\eta=\frac{\delta(2\delta-1)}{(2\delta+1)(4-\delta)} >0 $. \vspace{-2mm}
\end{description} Here $c_1,c_2>0$ are absolute constants. }
\end{kqr}

In the current paper, adapting the method of proof in \cite{hjk} together with the sharper discrepancy estimate coming from Theorem \ref{thm1}, we are able to obtain a slight improvement to the result of Haynes, Jensen and Kristensen: 

\begin{thm} \label{thm3}  Let $\varepsilon>0$ be fixed and $(\alpha_j)_{j=1}^{\infty}\subseteq \Bad$ be a sequence of badly approximable numbers. There exists a subset $G\subseteq\Bad$ of Hausdorff dimension $\dim G=\dim\Bad=1$ such that for any $\beta\in G,$ the following holds: 
\begin{equation}
q \|q\alpha_j\| \|q\beta-\gamma\| < \frac{(\log_3 q)^{1/2 + \varepsilon}}{(\log q)^{1/2}} \hspace{5mm} \text{ for inf. many  } q=1,2,\ldots
\end{equation}
for all $j=1,2,\ldots$ and for all $\gamma\in [0,1)$. 
\end{thm}

As another consequence of Theorem \ref{thm1}, we obtain a statement regarding the Fourier dimension of the set of exceptions to Philipp's theorem, namely the set of $x\in[0,1)$ for which \eqref{philequation} fails. Given a subset $F\subseteq[0,1)$, the {\em Fourier dimension} of $F$ is defined by
\begin{equation*}
\dim_F F  := \sup\left\{0\leq \eta\leq 1:  \exists \, \mu \in M_1^+(F)\, \text{ {\rm with }} \hspace{1mm}  \widehat{\mu}(t)= O(|t|^{-\eta/ 2}),\, |t|\to\infty \  \right\}. \vspace{2mm}
\end{equation*}
where $M_1^+(F)$ denotes the set of all positive Borel probability measures with support in  $F$.  A result of Frostman states that whenever a set $F\subseteq[0,1)$ supports a probability measure $\mu$ such that $\widehat{\mu}(t)= O(|t|^{-\eta/ 2})$,  then its Hausdorff dimension is $\dim F  \geq \min \{1, \eta\} $ (see \cite[Chapter 4]{falc}). Therefore, the Fourier dimension of an arbitrary set is bounded above by the Hausdorff dimension: $\dim_F F \leq \dim F$.

\begin{cor}\label{cor} Let $(n_k)_{k=1}^{\infty}$ be a lacunary sequence of integers and consider the set 
\begin{equation*}
A \, = \, \left\{x\in [0,1): \limsup_{N\to\infty} N^{1/2}(\log\log N)^{-1/2}D_N(n_kx)=0 \text{ or } \infty  \right\} 
\end{equation*}
of $x\in [0,1)$ for which the conclusion \eqref{philequation} in Philipp's Theorem fails. Then the Fourier dimension of $A$ is $\dim_F A=0$.
\end{cor}

The proof of Corollary \ref{cor} is straightforward: suppose $\dim_FA>0$. Then $A$ supports some probability measure $\mu$ such that $\widehat{\mu}(t) = O(|t|^{-\eta/2}),\, |t|\to\infty$ for some $\eta>0$, and by Theorem \ref{thm1}, \eqref{thm1eq} holds for $\mu$--almost all $x\in A$, which contradicts the definition of $A$.

However, regarding the Hausdorff dimension of the set $A$ defined in Corollary \ref{cor}, it can be deduced that $\dim A=1$. Indeed, given a lacunary sequence $(n_k)_{k=1}^{\infty}$, a result obtained independently by Pollington \cite{poll} and de Mathan \cite{demathan} states that the set $B := \{x\in[0,1) : (n_kx)_{k=1}^{\infty} \text{ is not dense}  \mod 1\}$ has Hausdorff dimension $1$. Thus $\dim A = 1$ as well, since $B\subseteq A$.

%

\section{Proof of Theorem \ref{thm1}}
\subsection{The upper bound}
\noindent We employ a classical method of proof for discrepancy estimates, which has been used by Philipp \cite{philipp}, Erd\H{o}s $\&$ Gal \cite{erdosgal} and Gal $\&$ Gal \cite{galgal}. For any positive integer $N$, the discrepancy $D_N(n_kx)$ satisfies  
\begin{equation*}
ND_N(n_kx) = \sup_{0\leq \alpha<\beta<1}\left|\sum_{k\leq N}\chi_{[\alpha,\beta)}(n_kx)-N(\beta-\alpha) \right| \leq 2\sup_{f\in \ca{F}}\left|\sum_{k\leq N}f(n_kx) \right|,
\end{equation*}
where $\ca{F}$ denotes the set of functions $f:[0,1)\rightarrow\mathbb{R}$ which are $1$-periodic with $\int_{0}^{1}\!f(x)\mathrm{d}x=0$ and have bounded variation. Moreover, since every such function is trivially the sum of an even and an odd function with the same properties, we may restrict our attention to the set $\ca{F}^*\subseteq\ca{F}$ of functions $f\in\ca{F}$ which are additionally even. 
\subsubsection{Some auxiliary results}
\noindent Let $f$ be an even function of bounded variation on $[0,1]$ such that
\begin{equation} \label{fassumptions}
\text{Var}f \leq 2, \hspace{5mm} \|f\|_{\infty}\leq 1, \hspace{5mm} f(x+1) = f(x) \hspace{5mm} \text{ and } \hspace{5mm} \int_{0}^{1}f(x)\mathrm{d}x = 0 \,  
\end{equation} and let $$f(x)\, = \, \sum_{j=1}^{\infty}2c_j\cos(2\pi jx)\, = \, \sum_{|j|=1}^{\infty}c_j\, e(jx) $$ be its Fourier series expansion. Observe that $c_j=c_{-j}$ for all $j$ and \eqref{fassumptions} imposes $c_0 = 0$ and $c_j \leq |j|^{-1}$. We set 
$$ f_n(x) = \sum_{1\leq |j|\leq n}\hspace{-2mm}c_j\, e(jx) \, .$$
\noindent Write $r=q^{1/2}$ (where $q>1$ is as in \eqref{lacunary}) and define
\begin{eqnarray*}
\phi_n(x) = f(x) - f_n(x), \\[1ex]
\psi(x;m) = \sum_{r^m\leq |j|<r^{m+1}}\hspace{-5mm}c_j\,e(jx), \\[1ex]
\Phi_N(x;m) = \sum_{k\leq N}\psi(n_kx;m) .  
\end{eqnarray*}
Also for $f\in L_2(\mu)$ we put $\|f\|_{L_2(\mu)}^2 =  \int |f|^2\mathrm{d}\mu.$ In the following lemmas we calculate the $\|\cdot \|_{L_2(\mu)}$--norm of sums of the form $\sum_{k\leq N}\phi_T(n_kx)$, first by calculating the norm of the blocks $\Phi_N(x;m)$ in Lemma \ref{lem1}, and then combining these estimates in Lemma \ref{lem2}. 
\begin{lem} \label{lem1} We have 
$$ \int|\Phi_N(x;m)|^2\mathrm{d}\mu(x) \, \ll \, \frac{N}{r^{m\eta}} \,   , $$
where $\eta>0$ is as in \eqref{mudecay}.
\end{lem}
\begin{proof}
We calculate
\begin{eqnarray*}
\int|\Phi_N(x;m)|^2\mathrm{d}\mu(x) &=& \mathop{\mathop{\sum\sum}_{1\leq k,l\leq N}}_{r^m\leq |i|,|j|<r^{m+1}}\hspace{-4mm}\int c_i c_j e((in_k-jn_l)x)\mathrm{d}\mu(x) \\
& \ll &  \mathop{\mathop{\sum\sum}_{1\leq k,l \leq N}}_{r^m\leq |i|,|j|<r^{m+1}}\hspace{-3mm}c_i c_j \delta_{in_k, jn_l} +   
\mathop{\mathop{ \mathop{\sum\sum}_{1\leq k,l\leq N}}_{r^m\leq |i|,|j|<r^{m+1}} }_{|in_k-jn_l| \geq 1}\hspace{-3mm}\frac{c_i c_j}{|in_k-jn_l|^{\eta}}.    
\end{eqnarray*}
Regarding the first of these terms, we can show as in \cite[Lemma $1$]{philipp} that it has order of magnitude 
$$ \ll N\log\left(\frac{r^{m+1}}{r^m}\right)\frac{1}{r^m} \ll \frac{N}{r^m}\,  . $$
The second term is 
\begin{equation}\label{asum}  \mathop{\mathop{ \mathop{\sum\sum}_{1\leq k,l\leq N}}_{r^m\leq |i|,|j|<r^{m+1}} }_{|in_k-jn_l| \geq 1}\hspace{-4mm}\frac{c_i c_j}{|in_k-jn_l|^{\eta}}  \ll \hspace{-4mm}\mathop{\mathop{\mathop{\sum\sum}_{1\leq k\leq N}}_{r^m\leq |i|,|j|<r^{m+1}}}_{i \neq j}\hspace{-4mm}c_ic_j |n_k(i-j)|^{-\eta} + \hspace{-4mm} \mathop{\mathop{ \mathop{\sum\sum}_{1\leq l<k\leq N}}_{r^m\leq |i|,|j|<r^{m+1}} }_{|in_k-jn_l| \geq 1}\hspace{-4mm}\frac{c_i c_j}{|in_k-jn_l|^{\eta}} \, \cdot
\end{equation}
\noindent The first of the sums in the right hand side of \eqref{asum} is at most
\begin{eqnarray*}
\mathop{\mathop{\mathop{\sum\sum}_{1\leq k\leq N}}_{r^m\leq i,j<r^{m+1}}}_{i\neq j} \hspace{-4mm}c_{i}c_{j}\left(n_{k}\left|i-j\right|\right)^{-\eta} & = & 2\left(\sum_{k\leq N}n_{k}^{-\eta}\right)\mathop{\sum\sum}_{r^m\leq i<j<r^{m+1}}\hspace{-4mm}c_{i}c_{j}\left|i-j\right|^{-\eta} \\[-2.6ex]
 & \ll & \mathop{\sum\sum}_{r^m\leq i<j<r^{m+1}} (c_i^2 + c_j^2) \left(j-i\right)^{-\eta}\\ 
 & \ll & \sum_{r^m\leq i<r^{m+1}}\frac{1}{i^{2}}\sum_{ j=1}^{r^{m+1}-i}\frac{1}{j^{\eta}}\\[1ex]
 & \ll & \sum_{r^m\leq i<r^{m+1}}\frac{1}{i^{2}}\left(r^{m+1}-i\right)^{1-\eta}\\[1ex]
 & \ll & \sum_{i=r^m}^{\infty}\frac{1}{i^{2}}r^{(1-\eta)m} \\[1ex] 
 & \asymp & \frac{1}{r^{ m\eta }} \cdot
\end{eqnarray*}
\noindent Regarding the second sum in the right hand side of \eqref{asum}, under the conditions of summation we get \vspace{-2mm}
\begin{eqnarray*}
|in_k - jn_l| =|i|n_k\left|1-\frac{jn_l}{in_k}\right| \geq |i|n_k\left|1-\frac{1}{r}\right| \gg r^mn_k \, ,
\end{eqnarray*}
whence 
\begin{eqnarray*}
\mathop{\mathop{\sum\sum}_{1\leq l<k\leq N}}_{r^m\leq |i|,|j|< r^{m+1}}\hspace{-3mm}c_i c_j |in_k-jn_l|^{-\eta} & \ll & \sum_{l< N}\sum_{i=r^m}^{r^{m+1}}\frac{1}{i^2} \sum_{k=l+1}^{N}\sum_{j=r^m}^{r^{m+1}}\frac{1}{r^{m\eta}\,n_k^{\eta}} \\[-1.5ex]
 & \ll & \sum_{l\leq N}\sum_{i=r^m}^{r^{m+1}}\frac{1}{i^2} \cdot \frac{r^m}{r^{m\eta}\,n_l^{\eta}}\\
 & \ll & \frac{1}{r^{m\eta }} \cdot
\end{eqnarray*}
 

\end{proof}

\begin{lem} \label{lem2} 
For any positive integer $N,T$ large enough we have
\begin{equation*}
\int\!\left(\sum_{k\leq N}\phi_T(n_kx) \right)^2\!\mathrm{d}\mu(x)\, \ll \,  \frac{N}{T^{\eta}} \, \cdot
\end{equation*}
\end{lem}
\begin{proof}
Let $m_0$ be the positive integer such that $r^{m_0-1}\leq T < r^{m_0}$. 
By observing that 
\begin{eqnarray*}
\phi_T(n_kx) & = & \sum_{|j|=T}^{r^{m_0}-1}c_j e(jn_k x) + \sum_{m=m_0}^{\infty} \psi(n_kx;m),
\end{eqnarray*}
we conclude that
\begin{eqnarray*}
\Big\|\sum_{k\leq N}\phi_T(n_kx) \Big\|_{L_2(\mu)} & \leq & \Big\|\sum_{k\leq N}\sum_{|j|=T}^{r^{m_0}-1}c_j e(jn_k x) \Big\|_{L_2(\mu)} + \sum_{m=m_0}^{\infty}\|\Phi_N(x;m)\|_{L_2(\mu)}.
\end{eqnarray*}
The first of the two terms is, by exploiting 
the arguments of the proof of Lemma \ref{lem1}, 
seen to be $\ll\, N^{1/2}T^{-\eta/2}$. Furthermore, the second term is, due to Lemma \ref{lem1}, up to a constant at most 
\begin{equation*}
\sum_{m=m_0}^{\infty} \frac{N^{1/2}}{r^{m\eta/2}} \, \ll \, \frac{N^{1/2}}{r^{m_0\eta/2}} \, \asymp \, \frac{N^{1/2}}{T^{\eta/2}}\, \cdot
\end{equation*}
\noindent Hence the result of the Lemma is shown.
\end{proof}

\noindent In what follows  $H, P$ and $T$ denote positive integers. We set 
\begin{equation*} g(x) = \sum_{|j|=1}^{T}c_j\,e(jx), \hspace{8mm} U_m(x) = \sum_{k=Hm+1}^{H(m+1)}g(n_kx)
\end{equation*}
Similar to inequality $(2.6)$ of \cite{philipp} we can write 
\begin{equation}\label{gnorm}
\|g\|_{\infty} \, \leq \, \text{\rm Var} f + \|f\|_{\infty} \, \leq \, 3 .
\end{equation}

\begin{lem} \label{lemma3} Let $0<\kappa <1$ and assume the integers $P,H,T$ are such that 
\begin{equation}\label{kprel}
\kappa\leq H^{-10/7},\hspace{6mm} 4PT \leq q^{2H} \hspace{4mm} \text{ and } \hspace{4mm} 3T^2H^2 < q^{\eta H} \, .
\end{equation}
Then for any $\delta>0$ we have 
\begin{equation*}
\int\exp\left(\kappa \sum_{m=0}^{P-1}U_{2m}(x) \right)\mathrm{d}\mu(x) \leq \exp\left(\textstyle\frac{1}{2}(1+2\delta)C_0\kappa^2\|f\| H (P+1)\right) \vspace{-0.5mm}
\end{equation*}
and
\begin{equation*}
\int\exp\left(\kappa \sum_{m=1}^{P}U_{2m-1}(x) \right)\mathrm{d}\mu(x) \leq \exp\left(\textstyle\frac{1}{2}(1+2\delta)C_0\kappa^2\|f\| H (P+1)\right),
\end{equation*}
where $C_0 = \frac12 + 2(q^{1/2}-1)^{-1} $ is an absolute constant. 
\end{lem}
\begin{proof}
We shall employ the inequality 
\begin{equation} \label{expineq}
|e^{z}| \leq 1 + z + \textstyle\frac{1}{2}(1+\delta)z^2 \, ,
\end{equation}
which is valid for all  numbers with $|z|<z_0(\delta).$ Since 
\begin{equation}
|\kappa U_{2m}(x)| = \left| \kappa \sum_{k=2Hm+1}^{H(2m+1)}\hspace{-3mm} g(n_kx)\right| \leq \kappa H\|g\|_{\infty} \stackrel{\eqref{gnorm}, \eqref{kprel}}{<} 1, \end{equation}
we can apply \eqref{expineq} to obtain
\begin{eqnarray*}
\exp\left(\kappa \sum_{m=0}^{P-1}U_{2m}(x) \right)\,  =\, \prod_{m=0}^{P-1}\exp(\kappa U_{2m}(x) )\, \leq\,  \prod_{m=0}^{P-1} \left(1 + \kappa U_{2m}(x) + \textstyle\frac{1}{2}(1+\delta)\kappa^2 U_{2m}^2(x)\right)\, .
\end{eqnarray*}
Observe that 
$$ U_{2m}(x)\, =\, \sum_{k=2mH+1}^{H(2m+1)}g(n_kx) \, = \,  \mathop{\mathop{\sum\sum}_{2mH<k< (2m+1)H}}_{1\leq|j|\leq T}\hspace{-4mm}c_j\, e(jn_kx) $$
is a sum of trigonometric terms of frequencies at least $n_{2mH} \geq q^{2mH}$ in absolute value. Write 
\begin{eqnarray*}
U_{2m}^2(x) &=&  \mathop{\mathop{\sum\sum}_{2mH<k< (2m+1)H}}_{1\leq|j|\leq T}\hspace{-4mm} c_j^2\, e(2jn_kx) \, + \, \mathop{\mathop{\sum\sum}_{2mH<k\leq (2m+1)H}}_{1\leq |j_1|,|j_2|\leq T}\hspace{-3mm}c_{j_1}\, c_{j_2} e \left((j_1+j_2)n_k x \right) \\ 
& & + 2  \mathop{\mathop{\sum\sum}_{2mH<k<l\leq (2m+1)H}}_{1\leq|j_1|,|j_2|\leq T}\hspace{-3mm}c_{j_1}c_{j_2}\, e\left((j_1n_k+j_2n_l) x \right) \\[2ex]
& = & W_{2m}(x) + V_{2m}(x) \, , 
\end{eqnarray*}
where 
\begin{eqnarray*}
W_{2m}(x) &: = &  \mathop{\mathop{\sum\sum}_{2mH<k< (2m+1)H}}_{1\leq|j|\leq T} c_j^2\, e(2jn_kx) \, + \,   \mathop{\mathop{\sum\sum}_{2mH<k\leq (2m+1)H}}_{1\leq |j_1|,|j_2|\leq T}\hspace{-3mm}c_{j_1}\, c_{j_2} e \left((j_1+j_2)n_k x \right) \\
& & + 2  \mathop{\mathop{\mathop{\sum\sum}_{2mH<k<l\leq (2m+1)H}}_{1\leq|j_1|,|j_2|\leq T}}_{|j_1n_k + j_2n_l|\geq n_{2mH}}\hspace{-3mm}c_{j_1}c_{j_2}\, e\left((j_1n_k+j_2n_l) x \right)
\end{eqnarray*}
\noindent is the sum of trigonometric terms appearing in $U_{2m}^2(x)$ with frequencies at least $n_{2mH}$, and
\begin{equation*}
V_{2m}(x) = U_{2m}^{2}(x) - W_{2m}(x) 
\end{equation*}
is the sum of the remaining terms in $U_{2m}^2(x)$, which have frequencies strictly less than $n_{2mH}$. It is shown in \cite{takahashi} and \cite[p.246]{philipp} that
\begin{equation*}
|V_{2m}(x)| \leq C_0 \, \|f \| H \, .
\end{equation*}
Hence 
\begin{equation*}
\int\exp\left(\kappa \sum_{m=0}^{P-1}U_{2m}(x) \right)\mathrm{d}\mu(x) \leq \int h(x) \mathrm{d}\mu(x),
\end{equation*} where we define the integrand to be 
\begin{eqnarray*}
h(x) &:=& \prod_{m=0}^{P-1}\left(1 + \textstyle\frac{1}{2}(1+\delta)C_0\kappa^2\|f\|H + \kappa U_{2m}(x) +\textstyle\frac{1}{2}(1+\delta) \kappa^2W_{2m}(x) \right) \\[2ex]
& = & \left( 1+ \textstyle\frac{1}{2}(1+\delta)C_0  \kappa^2\|f\|H\right)^P + \\
& & + \, \left(1+ \textstyle\frac{1}{2}(1+\delta)C_0  \kappa^2\|f\|H\right)^{P-1}\left(\kappa\sum_{m=0}^{P-1}U_{2m}(x) +  \textstyle\frac{1}{2}(1+\delta)\displaystyle\kappa^2\sum\limits_{m=0}^{P-1}W_{2m}(x) \right)\\
& & +\, \left(1+ \textstyle\frac{1}{2}(1+\delta)C_0  \kappa^2\|f\|H\right)^{P-2}\hspace{-4mm}\mathop{\mathop{\sum}_{0\leq m_1< m_2 < P}}\displaystyle\prod_{i=1}^{2}(\kappa U_{2m_i}(x) + \textstyle\frac{1}{2}(1+\delta)\kappa^2 W_{2m_i}(x))  \\
& & + \ldots \\
& & +\, \left(1 + \textstyle\frac{1}{2}(1+\delta)C_0  \kappa^2\|f\| H\right)^{P-s}\hspace{-6mm}\mathop{\sum}_{0\leq m_1<\ldots< m_s < P}\displaystyle\prod_{i=1}^{s}(\kappa U_{2m_i}(x) +\textstyle\frac{1}{2}(1+\delta)\kappa^2 W_{2m_i}(x))  \\
& & + \, \ldots  + \, \prod_{m=0}^{P-1}(\kappa U_{2m}(x) + \textstyle\frac{1}{2}(1+\delta)\kappa^2 W_{2m}(x) )\, .
\end{eqnarray*}
If we look at the $s$-th term in the above expansion, every factor 
$$ \kappa U_{2m_i}(x) + \textstyle\frac{1}{2}(1+\delta)\displaystyle\kappa^2 W_{2m_i}(x) $$
is a sum of trigonometric terms, which have frequencies lying between $n_{2m_iH}$ and $2Tn_{(2m_i+1)H}$ in absolute value. The number of these terms is at most $3T^2H^2$. Thus any product 
$$ \prod_{j=1}^{s}\left( \kappa U_{2m_j}(x) + \textstyle\frac{1}{2}(1+\delta) \displaystyle\kappa^2 W_{2m_j}(x) \right) $$
is a sum of at most $3^sT^{2s}H^{2s}$ trigonometric terms, each of them being multiplied by a coefficient at most $\kappa^s$ and having frequency which is at least 
\begin{eqnarray*}
 n_{2m_sH} - 2T(n_{2m_1H} + \ldots + n_{2m_{s-1}H}) &=& n_{2m_sH}\left[ 1 -2T\left( \frac{n_{2m_1H}}{n_{2m_sH}} + \ldots + \frac{n_{2m_{s-1}H}}{n_{2m_sH}}\right)\right] \\[1ex]
 & \geq & q^{2m_sH} \left[1 - 2T\left(\frac{1}{q^{2(m_s-m_1)H}}+ \ldots + \frac{1}{q^{2(m_s-m_{s-1})H}}  \right) \right] \\[1ex]
 & \geq & \frac{1}{2}q^{2m_sH}, 
\end{eqnarray*}
where we used the fact that
\begin{equation*}
\frac{1}{q^{2(m_s-m_1)H}}+ \ldots + \frac{1}{q^{2(m_s-m_{s-1})H}} \, \leq \, \frac{P}{q^{2H}} \, \stackrel{\eqref{kprel}}\leq \, \frac{1}{4T} \, \cdot
\end{equation*} 
We deduce that  
\begin{equation*}
\int  \prod_{j=1}^{s}\left( \kappa U_{2m_j}(x) + \kappa^2 W_{2m_j}(x) \right)  \mathrm{d}\mu(x)\, \leq \, 2^{\eta}\frac{(2\kappa T^2 H^2)^s }{q^{2m_s\eta H} }\, \stackrel{\eqref{kprel}}{\leq}\, \frac{2^{\eta}}{q^{m_s\eta H}}
\end{equation*}
and 
\begin{equation*}
\int\!\mathop{\sum }_{0\leq m_1<\ldots< m_s < P }\prod_{j=1}^{s}\left( \kappa U_{2m_j}(x) + \kappa^2 W_{2m_j}(x) \right) \mathrm{d}\mu(x)\, \leq\, 2^{\eta}\hspace{-4mm}\mathop{\sum}_{0\leq m_1<\ldots< m_s<P}\hspace{-3mm}q^{-m_s\eta H} \leq \frac{2^{\eta}q^{\eta}}{q^{\eta}-1} \, \cdot
\end{equation*}
Thus
\begin{eqnarray*}
\int h(x) \mathrm{d}\mu(x) & \leq & \frac{2^{\eta}q^{\eta}}{q^{\eta}-1}\sum_{s=0}^{P}\left( 1+\textstyle\frac{1}{2}(1+\delta) C_0\kappa^2\|f\|H\right)^{P-s}  \\[1ex]
& \leq &  \frac{2^{\eta}q^{\eta}}{q^{\eta}-1}\left( 1+ \textstyle\frac{1}{2}(1+\delta)C_0\kappa^2\|f\|H\right)^{P+1} \\[2ex]
& \leq & \frac{2^{\eta}q^{\eta}}{q^{\eta}-1} \exp(\textstyle\frac{1}{2}(1+\delta)C_0\kappa^2\|f\| H(P+1) )  \\[2ex]
& \leq & \exp(\textstyle\frac{1}{2}(1+2\delta)C_0\kappa^2\|f\| H(P+1) ).
\end{eqnarray*}
The second inequality of the lemma follows in precisely the same way.
\end{proof}

\begin{prop}\label{prop} Let $M_0\geq 0,\, M\geq 1$ be positive integers and let $R\geq 1$ be a real number. Assume $f$ satisfies \eqref{fassumptions} and $\|f\| \geq M^{-3/5}$. Consider the set 
\begin{equation}
A = \left\{ x\in [0,1) \, : \, \left|\sum_{k=M_0+1}^{M_0+M}\!\!\!f(n_kx) \right| \geq (1+4\delta)C_0R\|f\|^{1/4}(M\log\log M)^{1/2} \right\} ,
\end{equation}
where $C_0>0$ is the constant from Lemma \ref{lemma3}. Then 
\begin{equation}
\mu(A) \ll \, \exp\left(-(1+2\delta)C_0\|f\|^{-1/2}R\log\log M\right)\, +\, \frac{1}{\|f\|^{1/2} R^2M^{4}} \, \cdot
\end{equation}
\end{prop}
\begin{proof}
Without loss of generality, we may assume that $M_0=0$. We put $H=[M^{1/30}], \, T=M^{\lceil 4/ \eta \rceil }$ and set
\begin{equation}
Q= 3C_0\|f\|^{1/4}R(M\log\log M)^{1/2}, \hspace{6mm} \kappa = \left(\|f\|^{-3/2}M^{-1}\log\log M \right)^{1/2}\, .
\end{equation}
We choose a positive integer $P$ such that
\begin{equation} \label{P}
H(2P+1) \, \leq \, M \, \leq \, H(2P+3)\, .
\end{equation}
Observe that $A\subseteq A_1 \cup A_2$, where 
\begin{equation*}\label{twosets}
A_1=\left\{x\in[0,1) : \sum_{k\leq M}g(n_kx) \geq (1+2\delta)Q  \right\},\,  A_2 = \left\{x\in[0,1) : \sum_{k\leq M}\phi_T(n_k x)\geq 2\delta Q \right\} .
\end{equation*}
We are going to give estimates for the measure of these sets using the Chebyshev--Markov inequality. In order to do that, we observe that
\begin{equation}\label{difference}
\kappa\left| \sum_{k\leq M}g(n_kx)-\sum_{m=0}^{2P}U_m(x)\right|\,\leq\,\kappa\hspace{-4mm}\sum_{k=H(2P+1)}^{M}\hspace{-3mm}|g(n_kx)|\,\stackrel{\eqref{gnorm},\eqref{P}}{\leq}\, 6\kappa H = o(1),\hspace{4mm} M\to\infty .
\end{equation}
Also since $\eqref{kprel}$ is satisfied, we can apply Lemma \ref{lemma3}. Regarding $A_1$, we estimate 
\begin{eqnarray*}
\mu(A_1) & \leq & e^{-(1+2\delta)\kappa Q}\int \exp\left(\kappa\sum_{k\leq M}g(n_kx)\right)\mathrm{d}\mu(x) \\[1ex]
& \ll & e^{-(1+2\delta)\kappa Q}\int\exp\left( \kappa\sum_{m=0}^{2P}U_m(x) \right)\mathrm{d}\mu(x) \hspace{4mm} \text{ (by \eqref{difference}) } \\[1ex]
& \leq & e^{-(1+2\delta)\kappa Q} \exp\left(2(1+2\delta)C_0\kappa^2\|f\|H(P+1)\right)\hspace{4mm}\text{(by Lemma \ref{lemma3})} \\[1ex]
& \leq & \exp\left(-(1+2\delta)C_0\|f\|^{-1/2}R\log\log M\right)\, ,
\end{eqnarray*}
while for $A_2$, the Chebyshev--Markov inequality again gives
\begin{eqnarray*}
\mu(A_2) &\ll & \frac{1}{Q^2}\int\left(\sum_{k\leq M}\phi_T(n_kx)\right)^2\mathrm{d}\mu(x) \\[1ex]
& \ll & \frac{1}{Q^2} M T^{-\eta} \hspace{8mm} \text{(by Lemma \ref{lem2})} \\[1ex]
& \ll & \frac{1}{\|f\|^{1/2} R^2M^{4}} \, \cdot 
\end{eqnarray*}
\end{proof}

\subsubsection{Proof of the upper bound}

\noindent Let $N\geq 1$ be a positive integer sufficiently large. We set
\begin{equation} \label{h1value}
H_1 \, = \, \left\lfloor \frac{\log N}{\log 4}\right\rfloor + 1\, .
\end{equation}
We define the functions $(\phi_{h}^{(j)})_{h\leq H_1}^{j\leq 2^h}$ as in \cite{philipp}. Under this notation, inequality $(3.2)$ in \cite{philipp} states that for each $0\leq\alpha<1$ there exists some index $j=j(\alpha)\leq 2^h$ such that
\begin{equation} \label{frels}
\sum_{h=1}^{H_1-1}\phi_{h}^{(j)}(x) \, \leq \, \chi_{[0,\alpha)}(x) \, \leq \, \sum_{h=1}^{H_1}\phi_{h}^{(j)}(x) .
\end{equation}

\noindent For $1\leq h \leq H_1,\, 1\leq j\leq 2^h, N\geq 1, M\geq 0$ set
$$ F(M,N,j,h;x) = \left| \sum_{k=M+1}^{M+N} \left( \phi_{h}^{(j)}(n_kx) - \int\phi_{h}^{(j)}(t)\mathrm{d}t \right)  \right|   $$

\noindent The following is a variation of Lemma $4$ in \cite{philipp}. The proof relies on a method of Gal $\&$ Gal, see \cite[Lemma $3.10$]{galgal}.
\begin{lem} Let $n$ be the positive integer such that $2^n \leq N < 2^{n+1}.$ There exist integers $(m_l)_{l=1}^{n}$ such that $0\leq m_l < 2^{n-l}, \, 1\leq l \leq n$ and
$$ F(0,N,j,h;x) \leq F(0,2^n,j,h;x) + \sum_{l=\lceil 5n/12\rceil}^{n}F(2^n+m_l2^l, 2^{l-1},j,h;x) + N^{5/12} .$$
\end{lem}

\noindent In what follows we set
\begin{equation*}
\chi(N) = 2(1+4\delta)C_0(N\log\log N)^{1/2} \, .
\end{equation*}

\noindent Define the sets 
\begin{eqnarray*}
G(n,j,h) &=& \{ 0\leq x < 1 : F(0,2^n,j,h;x)\geq 2^{-h/8}\chi(2^n)\},\\[2ex]
H(n,j,h,l,m) &=& \{ 0\leq x <1 : F(2^n+m2^l, 2^{l-1}, j, h;x) \geq 2^{-h/8}2^{(l-n-3)/6}\chi(2^n)\}, \\
G_n &=& \bigcup_{h\leq H}\bigcup_{j\leq 2^h}G(n,j,h), \hspace{4mm} H_n =  \bigcup_{h\leq H}\bigcup_{j\leq 2^h}\bigcup_{l=\lceil\frac{5n}{12}\rceil}^{n}\bigcup_{m\leq 2^{n-l}} H(n,j,h,l,m) \, .
\end{eqnarray*}

\begin{lem}Let $0<\delta_0<1$. There exists $n_0=n_0(\delta_0)\in\mathbb{N}$ such that 
$$ \mu\left(\bigcup_{n=n_0}^{\infty}(G_n\cup H_n) \right) < \delta_0 .  $$
\end{lem}
\begin{proof}
By \eqref{h1value} we have 
\begin{equation}
 N^{-1/2 }  \ll 2^{-(h+1)} \leq \left\| \phi_h^{(j)} - \int \phi_h^{(j)}(t)\mathrm{d}t \right\|^2 \leq 2^{-h},\hspace{5mm} 1\leq h \leq H_1,\, 1\leq j\leq 2^h.
\end{equation}
Applying Proposition \ref{prop} with  $M_0=0, M=2^n$, $R=1$ and $f=\phi_h^{(j)} - \int \phi_h^{(j)}(t)\mathrm{d}t $, we get
\begin{eqnarray*}
\mu\left( G(n,j,h)\right) & \ll & \exp\left(-(1+2\delta)C_0\|f\|^{-1/2}\log\log N\right) + \frac{1}{\|f\|^{1/2} N^{4}} \\
& \ll & \exp\left( -(1+2\delta)C_0 2^{h/4}\log\log 2^n \right) + \frac{2^{h/4}}{2^{4n}}, 
\end{eqnarray*}
hence
\begin{eqnarray*}
\mu(G_n)\, &\leq\,& \sum_{h\leq H_1}\sum_{j\leq 2^h} \mu\left( G(n,j,h)\right) \, \ll\,  \sum_{h\leq H_1}\sum_{j\leq 2^h}\left( (n\log 2)^{-(1+2\delta)C_0 2^{h/4} } + \frac{2^{h/4}}{2^{4n}} \right)\\  
&\ll & \sum_{h\leq H_1} \left( n^{-(1+2\delta)C_0 2^{h/4} } + \frac{2^{5h/4}}{2^{4n}} \right)  \, \ll \, n^{-(1+\delta)} \, .
\end{eqnarray*}
Applying Proposition \ref{prop} with $M_0=2^n + m2^l, \,  M=2^{l-1}, R=2^{(n-l)/3}$ and $f=\phi_h^{(j)} - \int \phi_h^{(j)}(t)\mathrm{d}t $ we obtain  
\begin{eqnarray*}
\mu\left( H(n,j,h,l,m)\right) & = & \mu\left( \{x: F(2^n+m2^l,2^{l-1},j,h) \geq 2^{-h/8} 2^{(l-n-3)/6}\chi(2^n)  \} \right) \\[1ex]
& \leq & \mu\left(\{x: F(2^n+m2^l,2^{l-1},j,h) \geq \|f\|^{1/4}\, R\, \chi(2^{l})  \} \right)  \\[1ex]
& \ll & \exp(-(1+2\delta)C_0 2^{h/4}2^{(n-l)/3}\log\log 2^l) \, + \, \frac{2^{h/4}}{2^{2n/3}2^{10l/3}} \, . 
\end{eqnarray*}
We now deduce that
\begin{eqnarray*}
\mu(H_n) & \ll & \sum_{h\leq H_1}\sum_{j\leq 2^h}\sum_{l=\lceil\frac{5n}{12}\rceil}^{n}\sum_{m=1}^{2^{n-l}}\mu(H(n,j,h,l,m))\\
& \ll & \sum_{h\leq H_1}\sum_{j\leq 2^h}\sum_{l=\lceil\frac{5n}{12}\rceil}^{n}  \left( \exp(-(1+2\delta)C_02^{h/4}2^{(n-l)/3}\log\log l  ) \, + \, \frac{2^{h/4}2^{n/3}}{2^{13l/3}} \right) \\
& \ll & \sum_{h\leq H_1}\sum_{j\leq 2^h}\sum_{l=\lceil\frac{5n}{12}\rceil}^{n} \left(n^{-(1+2\delta)C_0 2^{h/4}2^{(n-l)/3}}+  \frac{2^{h/4}2^{n/3}}{2^{13l/3}}\right) \\
& \ll & \sum_{h\leq H_1}\sum_{j\leq 2^h} \left(n^{-(1+2\delta)C_0 2^{h/4}}+  \frac{2^{h/4}}{2^{53n/36}}\right)\, \ll \,  n^{-(1+\delta)} \, .
\end{eqnarray*}
The conclusion of the Lemma is now  evident.
\end{proof}

\noindent We may now proceed to the final part of the proof of the upper bound. Choose an arbitrary $0\leq \alpha <1$. By \eqref{frels} we obtain 
\begin{eqnarray*}
\left|\sum_{k\leq N}\chi_{[0,\alpha)}(n_kx) -N\alpha  \right| &\leq & \sum_{h=1}^{H_1}\left|\sum_{k\leq N}\phi_{h}^{(j)}(n_kx) - N\int\phi_h^{(j)}(t)\mathrm{d}t  \right| + 2^{-H_1}N \\
& \leq & \sum_{h\leq H_1}\left( F(0,2^n,j,h)+\sum_{l=n/2}^{n}F(2^n+m_l2^{l-1}, 2^{l-1}.j,h) \right) + 2N^{1/2} \\
& \leq & \sum_{h\leq H_1} 2^{-h/8}\chi(2^n)\left( 1 +\sum_{l=n/2}^{n}2^{(l-n-3)/6} \right) + 2N^{1/2} \\
& \leq & (1+4\delta)(83+332(\sqrt{q}-1)^{-1})(N \log\log N)^{1/2}\, ,
\end{eqnarray*}
for all $0\leq x< 1$ lying outside a set of $\mu$--measure at most $\delta_0$. Hence for those $0\leq x<1$ we obtain for any $0\leq \alpha <\beta <1$
\begin{equation*}
\left|\sum_{k\leq N}\chi_{[\alpha,\beta)}(n_kx) - N(\beta-\alpha) \right| \leq (1+4\delta)(166+332(\sqrt{q}-1)^{-1})(N\log\log N)^{1/2}\, . 
\end{equation*}
Taking the supremum over all $0\leq \alpha < \beta<1$, we get
$$ ND_N(n_kx) \leq (1+4\delta) (166+332(\sqrt{q}-1)^{-1})(N\log\log N)^{1/2}  $$
for all $x$ in a set of $\mu$--measure at most $\delta_0$. Now letting $\delta\to 0$ and then $\delta_0\to 0$ we obtain the requested upper bound in Theorem \ref{thm1}.

\begin{rem} It is worth pointing out a minor oversight in Philipp's original proof \cite{philipp}. The reader can easily check that a correct application of Philipp's Proposition \cite[p.244]{philipp} to estimate $\lambda(H_n)$ leads to a bound of the form $\lambda(H(n,j,h,l,m))\ll 2^{7n/24},$ hence relation $(3.8)$ in \cite{philipp} is wrong. The problem can be overcome if we change the definition of the functions $g(x)$ in $(2.5)$ to $g(x)=\sum_{1\leq j\leq T}c_j\cos(2\pi jx),$ with $T$ a sufficiently large integer.  It is likely that this oversight is due to the fact that the symbol $N$ is used for both $D_N(n_kx)$ and $N=2^{l-1}$ in the application of Philipp's Proposition.
\end{rem}

\subsection{The lower bound   }

Given a sequence $(n_k)_{k=1}^{\infty}$, Koksma's Inequality implies that 
\begin{equation*}
ND_N(n_kx) \geq \frac{1}{4} \left|\sum_{k=1}^{N}e(n_kx) \right|, \hspace{5mm} x\in [0,1),
\end{equation*}
see \cite[p.143]{kuipers} for more details. Thus the lower bound in Theorem \ref{thm1} will follow immediately if we prove the following partial generalisation of the result of Erd\H{o}s and Gal in \cite{erdosgal}: 

\begin{prop} \label{muLIL}  Let $(n_k)_{k=1}^{\infty}$ be a lacunary sequence of integers such that \eqref{lacunary} is satisfied. If $\mu$ is a probability measure on $[0,1)$ such that $|\widehat{\mu}(t)| \ll |t|^{-\eta}, |t|\to\infty$ then 
\begin{equation}
\limsup_{N\to\infty}\frac{\left|\sum\limits_{k=1}^{N}e(n_kx) \right|}{\sqrt{N\log\log N}} \, \geq \, 1 \hspace{5mm} \text{ for $\mu$--almost all } x\in[0,1). 
\end{equation}
\end{prop}

\noindent The proof of Proposition \ref{muLIL} is essentially the same as in \cite{erdosgal}, with the only modifications being those relevant to the fact that $\mu$ is a probability measure other than the Lebesgue measure. We present here all steps of the proof which are essentially different and refer the reader to \cite{erdosgal} for the remaining parts.

\subsubsection{On the number of solutions of certain Diophantine inequalities}
In what follows the postive integers $p,N$ are fixed, the sequence $(n_k)_{k=1}^{\infty}$ is as in Theorem \ref{thm1} and 
$$ A(x,y) \equiv A(x_1,\ldots,y_p) = (x_1+\ldots+x_p) - (y_1+\ldots+y_p)$$
is a linear form in $2p$ variables which are allowed to take values in the set $\{n_1, \ldots , n_N\}.$

\begin{lem} \label{erdoslemma1} For any $0 < \alpha < \beta,$  
$$ \#\{1\leq k\leq N : \alpha \leq n_k \leq \beta \}  \leq  \frac{\log\left( \beta\alpha^{-1} q  \right)}{\log q} \cdot $$
\end{lem} 
\begin{proof} This is Lemma $1$ in \cite{erdosgal}. \end{proof}

\noindent In what follows, given $s\in\mathbb{R}$ and $r>0$ we write $B(s,r):=(s-r, s+r)$ for the interval with center $s$ and length $2r$.

\begin{lem} \label{erdoslemma2} For positive integers $N, K \geq 1$ and $s\in\mathbb{Z}$ we write $\phi(1,N;s,K)$ for the number of pairs $(n_k,n_l)$ with $1\leq k,l \leq N$ such that $n_k - n_l \in B(s, 2^{K-1})$ and $n_k\neq n_l$. Then
\begin{equation} \label{erdoseq0} \phi(1,N;s,K) \leq 2 \frac{\log^2(2^Kq^2(q-1)^{-1})}{\log^2q} \cdot 
\end{equation}
\end{lem}
\begin{proof} Since $\phi(1,N;s,K)=\phi(1,N;-s,K),$ we may assume without loss of generality that $s\geq 0$. If $0\leq s \leq 2^{K-1}$ then $\frac{1}{2}\phi(1,N;s,K)$ is at most equal to the number of pairs $(n_k,n_l)$ such that $1\leq n_k - n_l \leq 2^K$. If $(n_k,n_l)$ is such a pair, then 
\begin{equation*} 
2^{K} \geq n_k - n_l = n_k\left(1-\frac{n_l}{n_k}\right) \geq n_k(1-q^{-1}),
\end{equation*} 
hence $1 \leq n_k \leq 2^{K}(1-q^{-1})^{-1}$ and by Lemma \ref{erdoslemma1} the number of admissible $n_k$'s is at most $ \dfrac{\log(2^Kq^2(q-1)^{-1})}{\log q} \cdot $ Now we fix an admissible value of $n_k$ and we count the number of $n_l$'s which are acceptable for that specific $n_k$. Such $n_l$'s satisfy $1\leq n_l < n_k \leq 2^K(1-q^{-1})^{-1}$, so by Lemma \ref{erdoslemma1} their number is at most  $\dfrac{\log(2^Kq^2(q-1)^{-1})}{\log q}$. Hence the number of possible pairs $(n_k,n_l)$ is at most 
$$ \frac{\log^2(2^Kq^2(q-1)^{-1})}{\log^2q} . $$
On the other hand, if $s>2^{K-1}$ then $s+2^{K-1}\geq n_k - n_l \geq n_k(1-q^{-1})$ and $n_k \geq \max(1,s-2^{K-1}) = s-2^{K-1}$, hence $$s-2^{K-1} \leq n_k \leq (s+2^{K-1})(1-q^{-1})^{-1} $$ and by Lemma \ref{erdoslemma1} there are at most 
$$ \frac{1}{\log q}\log\left(\frac{s+2^{K-1}}{s-2^{K-1}} q(1-q^{-1})^{-1}\right) \leq \frac{\log(2^Kq^2(q-1)^{-1})}{\log q} $$
possible values for $n_k$. Regarding $n_l$, we get $n_k-s-2^{K-1} \leq n_l \leq n_k - s + 2^{K-1}$. Considering the cases $n_k\leq 2^{K-1}$ and $n_k > 2^{K-1}$ separately, Lemma \ref{erdoslemma1} gives at most $\dfrac{\log(2^Kq)}{\log q}$ values for $n_l$, and the number of pairs $(n_k,n_l)$ is again bounded above by  
$$ \frac{\log^2(2^Kq^2(q-1)^{-1})}{\log^2q} . $$
\end{proof}

\begin{lem} \label{erdoslemma3} For positive integers $N, K \geq 1$ and $s\in\mathbb{Z}$ let $\phi_p(s,K)$ be the number of pairs $(x_p,y_p)$ such that $  A(x,y) \in B(s,2^{K-1}),$ subject to the restrictions $x_1\leq \ldots \leq x_p$ and $y_1\leq \ldots \leq y_p$. Then  
\begin{equation} \label{erdoseq}\phi_p(s,K) \leq 4pN\frac{\log(2^{K+2}q)}{\log q} \cdot 
\end{equation}
\end{lem}
\begin{proof}Since $\phi_p(s,K)=\phi_p(-s,K)$, we may assume without loss of generality that $s\geq 0$.
First we count the number of requested pairs $(x_p,y_p)$ for which $x_p\leq 2s+ 2^{K}$. The assumptions imply that 
\begin{equation*} s- 2^{K-1} \leq (x_1 + \ldots + x_p) - (y_1 + \ldots + y_p) \leq px_p,\end{equation*}
hence $p^{-1}(s-2^{K-1}) \leq x_p \leq 2^{K+2} $. When $s>3\cdot2^{K-1}$, we have $\frac{1}{p}(s-2^{K-1})\leq x_p\leq 4(s-2^{K-1})$ and by Lemma \ref{erdoslemma1} there are at most $$ \frac{\log(4pq)}{\log q} \leq p\frac{\log(2^{K+2}q)}{\log q} $$ possible values for $x_p$. When $0\leq s\leq 3\cdot 2^{K-1},$ we have $1\leq x_p \leq 2s+2^K \leq 2^{K+2}$ and there are at most $$ \frac{\log(2^{K+2}q)}{\log q} \leq p\frac{\log(2^{K+2}q)}{\log q}$$ values of $x_p$. In both cases for $s$, there are at most $N$ choices for $y_p$, so the number of possible pairs $(x_p,y_p)$ with $x_p\leq 2^{K+2}$ is bounded above by $$pN \frac{\log(2^{K+2}q)}{\log q} \cdot $$
Next we count the number of requested pairs $(x_p,y_p)$ for which $x_p > 2^{K+2}$. The assumptions now imply that 
\begin{equation*} x_p \leq py_p +s+ 2^{K-1} \leq py_p + \frac{1}{2}x_p, \end{equation*}
hence $$ \frac{x_p}{2p} \leq y_p \leq px_p -s+ 2^{K-1} \leq px_p +2^{K-1} < 2px_p $$
and the number of possible values for $y_p$ is by Lemma \ref{erdoslemma1} at most $\dfrac{\log(4p^2q)}{\log q}$. Since there are at most $N$ possible choices for $x_p$, we have the upper bound 
$$ N \frac{\log(2p^2q)}{\log q} \leq 2Np \frac{\log(2^{K+2}q)}{\log q}$$ for the number of pairs $(x_p,y_p)$ with $x_p>2s+2^{K-1}.$ Combining the estimates for the two cases we obtain the requested bound \eqref{erdoseq}.

\end{proof}

\begin{lem}\label{erdoslemma4} For $1\leq p \leq N$ and $K\geq 1$ we write $\phi(p,N)$ for the number of solutions of $A(x,y)=0$ and $\phi(p,N;s,K)$ for the number of solutions of $ A(x,y) \in B(s,2^{K-1}) $, both subject to the restrictions $x_1\leq \ldots \leq x_p$ and $y_1\leq \ldots \leq y_p$. Then 
\begin{equation}  \label{erdoseq2}
\binom{N}{p} \leq \phi(p,N) \leq \binom{N}{p} + (cp)^p N^{p-1} 
\end{equation}
and
\begin{equation} \label{erdoseq3}
0 \leq \phi(p,N;s,K) \leq 2(4p)^{p-1}N^{p-1}\left( \frac{\log\left(2^{K+2}q^2(q-1)^{-1}\right)}{(\log q)^2}\right)^{p+1} \cdot
\end{equation}
\end{lem}
\begin{proof} Inequality \eqref{erdoseq2} is proved in Lemma $7$ of \cite{erdosgal}. In order to prove \eqref{erdoseq3}, we fix the value of $K$ and use induction on $p\geq 1$. For $p=1$,  \eqref{erdoseq3} is implied by \eqref{erdoseq0}. Now we assume \eqref{erdoseq3} is true for $1,2,\ldots, p-1$ and we seek an upper estimate for $\phi(p,N;s,K)$. To do this, we consider separately two sets of solutions:  First, those $2p$-tuples $(x_1,\ldots,y_p)$ with $x_1=y_1$. Then the number of tuples $(x_2,\ldots,x_p,y_2,\ldots,y_p)$ with $s-2^{K-1} \leq (x_2+\ldots+x_p)-(y_2+\ldots+y_p)\leq s+2^{K-1}$ is at most $\phi(p-1,N;s,K)$ and there are $N$ possible values for $(x_1,y_1)$, hence we have at most $$N\phi(p-1,N;s,K)$$ solutions of that kind. Next we consider $2p$-tuples with $x_1\neq y_1$. By \eqref{erdoseq} the number of $2(p-1)$-tuples $(x_2,\ldots,x_p,y_2,\ldots,y_p)$ is at most
\begin{equation*} \left(4pN \frac{2^{K+2}q^2(q-1)^{-1}}{\log q}  \right)^{p-1} \, \cdot
\end{equation*} 
For each such $2(p-1)$-tuple, the number of acceptable pairs $(x_1,y_1)$ with $x_1\neq y_1$ is given by \eqref{erdoseq0} and is at most $$2\frac{\log^2(2^{K+2}q^2(q-1)^{-1})}{\log^2 q} \cdot$$ Combining the two cases, we obtain 
\begin{eqnarray*}
\phi(p,N;s,K) &\leq & N\phi(p-1,N;s,K) + 2\frac{\log^2(2^{K+2}q^2(q-1)^{-1})}{\log^2 q} \\
& \leq & 2(4(p-1))^{p-2}N^{p-1}\left( \frac{\log\left(2^{K+2}q^2(q-1)^{-1}\right)}{(\log q)^2}\right)^{p} + 2\frac{\log^2(2^{K+2}q^2(q-1)^{-1})}{\log^2 q} \\
& \leq & 2(4p)^{p-1}N^{p-1}\left( \frac{\log\left(2^{K+2}q^2(q-1)^{-1}\right)}{(\log q)^2}\right)^{p+1} \cdot
\end{eqnarray*} 
\end{proof}

Now we are able to prove the final goal of this subsection, which is giving an estimate for the number of solutions of equations of the form $A(x,y)\in B(s,2^{K-1})$. The result follows immediately from Lemma \ref{erdoslemma4} since each solution to the aforementioned equation under the restrictions $x_1\leq \ldots x_p, y_1\leq \ldots \leq y_p $ gives rise to $(p!)^2$ solutions $(x_1,\ldots,y_p)$.  

\begin{lem}For positive integers $1\leq p \leq N$, there exists a constant $c>0$ such that for any $s\in\mathbb{Z}$
\begin{equation} \label{erdoseq4}
(p!)^2\binom{N}{p} \leq \sum_{A(x,y)=0} 1  \leq (p!)^2\binom{N}{p} + (cp)^{3p}N^{p-1} 
\end{equation}
and
\begin{equation}\label{erdoseq5}
0 \leq \sum_{A(x,y)\in B(s,2^{K-1})}\hspace{-4mm}1 \leq \,2(4p)^{3p-1}N^{p-1}\left( \frac{\log\left(2^{K+2}q^2(q-1)^{-1}\right)}{(\log q)^2}\right)^{p+1} \cdot
\end{equation}
\end{lem}

\subsubsection{Metrical Estimates on Exponential Sums}

The previous Lemmas on the number of solutions of Diophantine equations with linear forms are used to estimate the moments of the function

$$ F(N;x) = \left| \sum_{k=1}^{N}e(n_kx)\right|, \hspace{5mm}  x\in [0,1). $$

\noindent For $p\geq 1$ and $0\leq \alpha < \beta \leq 1$ we need to estimate the integral 
$$ I_{2p} = \int_{\alpha}^{\beta}|F(N;x)|^{2p}\mathrm{d}\mu(x) . $$

\noindent This in turn is used to provide estimates for the function 
\begin{equation} \label{phi}
\phi(t) = \mu\left(\{x\in[\alpha,\beta] : F(N;x)\geq \sqrt{tN\log\log N} \} \right) \, .
\end{equation}

The following lemma shows that if a probability measure on $[0,1)$ has Fourier transform with polynomial decay rate, then the same is true for any restriction of this measure to some subinterval. 

\begin{lem} \label{decaylemma}
Let $\mu$ be a probability measure on $[0,1)$ and $B=(\alpha,\beta)\subseteq[0,1)$ be a subinterval with $\mu(B)>0$. Let $\nu$ be the probability measure defined by $\nu(A)= \frac{1}{\mu(B)}\mu(A\cap B)$ for any subset $A\subseteq [0,1)$. If \eqref{mudecay} holds for some $\eta>0$, then
\begin{equation} \label{decay}
\widehat{\nu}(t) \, \ll \, \frac{1}{\mu(B)} |t|^{-\eta}, \hspace{5mm} |t|\to\infty \, .
\end{equation}
\end{lem}

\noindent We include the proof of the Lemma in the Appendix at the end of the paper.

\begin{prop}\label{I2pestimate} If $\alpha, \beta$ are such that $\mu((\alpha,\beta))\geq 1/n_1^{\eta}\sqrt{N} $ and $1\leq p \leq 3\log\log N$, then 
\begin{equation}\label{I2p}
\left| I_{2p} - \mu((\alpha,\beta))p! N^p \right| \leq \mu((\alpha,\beta))N^{p-\frac{1}{2}} , 
\end{equation}
for all $N$ large enough. 
\end{prop}
\begin{proof}
By definition of $F(N;x)$ we have
\begin{eqnarray*}
I_{2p} &=& \mathop{\sum_{1\leq k_1,\ldots,k_p\leq N}}_{1\leq l_1,\ldots,l_p\leq N}\int_{\alpha}^{\beta}e\left(\sum_{j=1}^{p}(n_{k_j}-n_{k_l})x \right)\mathrm{d}\mu(x) \\[2ex]
& = & \mu((\alpha,\beta))\sum_{A(x,y)=0}\!1 \, + \,  \sum_{0<|A(x,y)|\leq \frac{1}{2}n_1}\int_{\alpha}^{\beta}e(A(x,y)t)\mathrm{d}\mu(t) \\[1ex]
&   & \hspace{2mm}+\, O\left( \sum_{K=1}^{\infty}\sum_{2^K\leq A(x,y) < 2^{K+1}}\int_{\alpha}^{\beta}e(A(x,y)t)\mathrm{d}\mu(t)   \, \right) ,
\end{eqnarray*}
where the implicit constant in the $O$-estimate is equal to $1$. The first term is estimated by \eqref{erdoseq4}, the second term again by \eqref{erdoseq4} and the trivial bound $|e(x)|\leq 1$, while for the third term we use \eqref{decay} and \eqref{erdoseq5}. Thus 
\begin{eqnarray*}
\left|I_{2p} - \mu((\alpha,\beta))(p!)^2\binom{N}{p}  \right| & \leq & \mu((\alpha,\beta))(cp)^{3p}+ 2\mu((\alpha,\beta))(4c_1p)^{3p-1}N^{p-1}  \\
& \, &   + \, \frac{2(4p)^{p-1}N^{p-1}}{n_1^{\eta}(\log q)^{2p+2}}\sum_{K=1}^{\infty} \frac{1}{2^{\eta K}}     \\
& \leq & 2\mu((\alpha,\beta))(cp)^{3p}N^{p-1} + c_2\frac{(4p)^{p-1}N^{p-\frac{1}{2}}}{n_1^{\eta}\sqrt{N}},
\end{eqnarray*}
where $c_1=\log(4q^2(q-1)^{-1})$ and $c_2>0$ is some constant depending only on $\eta$ and $q$. Using the inequality
$$\left| (p!)^2\binom{N}{p} -p!N^p\right| \leq \frac{1}{2}N^{p-\frac{1}{2}} $$
finally yields the requested estimate \eqref{I2p}.
\end{proof}

\noindent Armed with Proposition \ref{I2pestimate} the analogue of \cite[Lemma 8]{erdosgal} follows. The proof is omitted, as it involves precisely the same arguments. 

\begin{lem} Let $\phi$ be the function defined in \eqref{phi} and $0<\varepsilon<1$. Then
$$ \phi(1-\varepsilon) > \frac{\mu((\alpha,\beta))}{(\log N)^{1-4\varepsilon^2}} \, \cdot $$
\end{lem}

\noindent  The remaining steps for the proof of the lower bound in Theorem \ref{thm1} go along the lines of \cite[p. 77-80]{erdosgal}, with the appropriate modifications for the measure $\mu$ instead of the Lebesgue measure.  

\section{Proof of Theorem \ref{thm3}}

We utilize the discrepancy estimate coming from Theorem \ref{thm1} in order to improve the result in \cite{hjk}. We set
$$ \psi(N) = N^{-1/2}(\log\log N)^{\frac{1}{2}+\varepsilon}  \, . $$
For any $i\geq 1$ let $(q_n^{(i)})_{n=1}^{\infty}$ be the sequence of denominators associated with the continued fraction expansion of $\alpha_i,$ and set 
\begin{equation*}
G_i =\bigcap_{\gamma\in\mathbb{R}} \left\{\beta\in [0,1) : \|q_n^{(i)}\beta-\gamma \| \leq \frac{(\log_3 q_n^{(i)})^{\frac{1}{2}+\varepsilon}}{(\log q_n^{(i)})^{1/2}} \text{ for inf. many } n=1,2,\ldots \right\} .
\end{equation*}
The proof of the theorem will be complete as long as we show that the set
$$ G = \bigcap_{i=1}^{\infty}\, (G_i \cap \Bad ) \, \subseteq \,\Bad \, $$
has Hausdorff dimension $\dim G=1$. To that end, let $\mu=\mu(N,\delta,\varepsilon)$ be any of the probability measures as in Theorem of Kaufman and Queff\'elec-Ramar\'e. Now for each real $\gamma$ define the sequence of indices 
$$ N_{k}^{\gamma} = \min\left\{N\geq 1\, :\, \#\{1\leq n\leq N : \{q_n^{(i)}\beta\}\in B(\gamma,\psi(N)) \}=k   \right\}, \hspace{4mm} k\geq 1. $$
{\it Claim:} The sequence $(N_k^{\gamma})_{k=1}^{\infty}$ is well-defined. \newline
{\it Proof of Claim:} Since $\alpha_i$ is badly approximable, the sequence $(q_n^{(i)})_{n=1}^{\infty}$ is lacunary and also 
\begin{equation} \label{growthrate}
\log q_n^{(i)} \, \asymp \, n , \hspace{5mm} n\to\infty\,
\end{equation}
(see \cite[p.288, 297]{PV} for more details). Hence the lacunarity property together with Theorem \ref{thm1} imply that for finally all $N\geq 1$ we have
\begin{equation*}
D_N(q_k^{(i)}\beta) \leq C_iN^{-1/2}(\log\log N)^{\frac{1}{2} } \hspace{4mm} \text{ for $\mu$-almost all } \beta\in [0,1). 
\end{equation*} 
Here the constant $C_i>0$ depends on $\alpha_i \in\Bad$ as in Theorem \ref{thm1}. For these values of $\beta$ lying in a set of full $\mu$-measure, the definition of discrepancy yields 
\begin{equation*}
\left|\#\{1\leq k\leq N : \{q_k^{(i)}\beta\}\in B(\gamma,\psi(N)) \} -2N\psi(N) \right| \leq C_i(N \log\log N)^{1/2}
\end{equation*}
for all $\gamma\in\mathbb{R}$. Hence
\begin{eqnarray*}
\#\{1\leq k\leq N : \{q_k^{(i)}\beta\}\in B(\gamma,\psi(N))  \} & \geq & 2N\psi(N) -C_i(N\log\log N)^{1/2} \\
&\geq & N^{1/2}(\log\log N)^{\frac{1}{2}+\varepsilon},
\end{eqnarray*}
for all $N$ sufficiently large. This inequality shows that for all $k\geq 1$ there exists $N_k^{\gamma} \in \mathbb{N}$ such that 
$$ \#\{1\leq n\leq N_k^{\gamma} : \{q_n^{(i)}\beta\}\in B(\gamma,\psi(N_k^{\gamma})) \}=k.   $$
The claim is proved. 

\noindent Thus for all $k=1,2,\ldots$ and for all $\beta$ in a set of full $\mu$--measure  we have
\begin{eqnarray*}
\| q_{N_k^{\gamma}}^{(i)}\beta -\gamma \| & \leq & (N_k^{\gamma})^{-\frac{1}{2}}(\log\log N_k^{\gamma})^{\frac{1}{2}+\varepsilon} \\
 & \stackrel{\eqref{growthrate}}{\ll} & \frac{(\log_3 q_{N_k^{\gamma}}^{(i)} )^{\frac{1}{2}+\varepsilon}}{( \log q_{N_k^{\gamma}}^{(i)} )^{1/2}} \, \cdot
\end{eqnarray*}
Thus $\mu(G_i)=1$ for all $i=1,2,\ldots$ and hence $\mu(G)=1$. Since $\mu=\mu(N,\delta,\varepsilon)$ was arbitrarily chosen, by property (ii) of the theorem of Kaufman and Queff\'elec--Ramar\'e together with the mass distribution principle.
cf. \cite[p. 975]{beresnevichvelani}, we conclude that 
$$ \dim G \geq \delta, \hspace{4mm} \text{ for all  }\,  \textstyle\frac{1}{2}<\delta <1, $$   
so $\dim G =1$ as required.

\newpage
 
\section{Appendix : The decay rate of the Fourier transform of the restricted measure.} 
 
Here we present a proof of Lemma \ref{decaylemma}, since we have not been able to locate the statement in the literature. The proof follows the one of an analogous result in \cite[p. 252]{kahane}. 

\noindent Since \eqref{mudecay} holds, there exists a constant $C_1>0$ such that 
\begin{equation*}
|\widehat{\mu}(t)| \leq C_1(1+|t|)^{-\eta} \hspace{5mm} \text{for all }t\in\mathbb{R}.
\end{equation*}
Let $\phi:\mathbb{R}\rightarrow (0,\infty)$ be a $C^{\infty}$ function which is equal to $1$ on the interval $B=(\alpha,\beta)$. Since $\phi$ is $C^{\infty}$, we have 
$$ \phi(x) \, = \, \sum_{k=-\infty}^{+\infty} \widehat{\phi}(k)e(kx)\, , \hspace{5mm} x\in\mathbb{R}  $$
where the convergence is uniform for all $x\in\mathbb{R}$.  Furthermore, since $\phi$ is a $C^{\infty}$ function, there exists a constant $C_{\eta}>0$ such that 
\begin{equation} \label{phifourierdecay}
|\widehat{\phi}(k) | \leq C_{\eta} (1 + |k| )^{-(1+\eta)}, \hspace{4mm} k\in\mathbb{Z} .
\end{equation} Set $S=\sum\limits_{k=-\infty}^{+\infty}(1+|k|)^{-(1+\eta)}<\infty.$ For the probability measure $\nu$ defined as in Lemma \ref{decaylemma} we have for $t>0$
\begin{eqnarray*}
|\widehat{\nu}(t)| & = & \frac{1}{\mu(B)}\left|\int_{B}e(-tx)\mathrm{d}\mu(x) \right| \\
& \leq &  \frac{1}{\mu(B)}\left|\int e(-tx)\phi(x)\mathrm{d}\mu(x) \right| \\
& = & \frac{1}{\mu(B)}\left|\int e(-tx)\sum_{k=-\infty}^{+\infty}\widehat{\phi}(k)e(kx)\mathrm{d}\mu(x) \right| \\
& = & \frac{1}{\mu(B)} \left|\sum_{k=-\infty}^{+\infty}\widehat{\phi}(k) \int e(-(t-k))\mathrm{d}\mu(x) \right| \\
& \leq & \frac{1}{\mu(B)}\sum_{k=-\infty}^{+\infty}|\widehat{\phi}(k)\widehat{\mu}(t-k)| \\
& = & \frac{1}{\mu(B)}\sum_{|k|\leq \frac{1}{2}|t|}|\widehat{\phi}(k)\widehat{\mu}(t-k)|\, +\, \frac{1}{\mu(B)}\sum_{|k|>\frac{1}{2}|t|}|\widehat{\phi}(k)\widehat{\mu}(t-k)| \, .
\end{eqnarray*} 
We deal with the first of the two terms. The condition of summation implies that $\frac{1}{2}|t| \leq |t-k| \leq \frac{3}{2}|t|$. Hence employing \eqref{mudecay} and \eqref{phifourierdecay} we get
\begin{eqnarray*}
\frac{1}{\mu(B)}\sum_{|k|\leq \frac{1}{2}|t|}|\widehat{\phi}(k)\widehat{\mu}(t-k)| &\leq & \frac{1}{\mu(B)}\sum_{|k|\leq \frac{1}{2}|t|}C_{\eta}(1+|t|)^{-(1+\eta)}\left(1+\frac{1}{2}|t|\right)^{-\eta}\\ &\leq & \frac{1}{\mu(B)}2^{\eta}C_{\eta}S(1+|t|)^{-\eta} \, .
\end{eqnarray*}
Regarding the second term, using the trivial bound $|\widehat{\mu}(t)|\leq 1$ together with \eqref{phifourierdecay} we get 
\begin{eqnarray*}
\frac{1}{\mu(B)}\sum_{|k|>\frac{1}{2}|t|}|\widehat{\phi}(k)\widehat{\mu}(t-k)| \leq \frac{1}{\mu(B)}\sum_{|k|>\frac{1}{2}|t|}C_{\eta}(1+|k|)^{-(1+\eta)} \leq \frac{1}{\mu(B)}C_{\eta}2^{\eta}(1+|t|)^{-\eta}.
\end{eqnarray*}
Combining the two estimates, we obtain
$$ |\widehat{\nu}(t)| \leq \frac{1}{\mu(B)}C(1+|t|)^{-\eta} \hspace{4mm} \text{for all }t>0$$ 
with $C=2^{\eta}C_{\eta}(1+S)>0$. The same bound is also true for all real values of $t$ in view of the relation $|\widehat{\nu}(t)| = |\widehat{\nu}(-t)|$, hence the Lemma is proved. \vspace{2mm}

\noindent \textbf{Acknowledgements: }We would like to thank Professor C. Aistleitner and Professor S. Velani for suggesting this direction of research and for many useful comments during the preparation of this paper. We also thank Bence Borda for pointing out Corollary \ref{cor} to us.

\noindent Niclas Technau: Department of Mathematics,
University of York,

\vspace{-2mm}

\noindent\phantom{Niclas Technau: }Heslington, York, YO10
5DD, England.
\vspace{-2mm}

\noindent\phantom{Niclas Technau: }e-mail: niclastechnau@gmail.com
\vspace{2mm}

\noindent Agamemnon Zafeiropoulos: Institute of Analysis and Number Theory,
TU Graz,

\vspace{-2mm}

\noindent\phantom{Agamemnon Zafeiropoulos: }Steyrergasse 30/II, 8010 Graz, Austria.

\vspace{-2mm}

\noindent\phantom{Agamemnon Zafeiropoulos: }e-mail: zafeiropoulos@math.tugraz.at

\end{document}